\documentclass[12pt]{article}
\usepackage{mathrsfs}
\usepackage{amsfonts}
\usepackage[all]{xy}
\usepackage{amsmath,amssymb}
\usepackage{amscd}

\baselineskip 5.5pt \lineskip 5.5pt \numberwithin{equation}{section}

\def\qed{\hfill$\Box$\par}

\def\qed{\ \ \ifhmode\unskip\nobreak\fi\ifmmode\ifinner
         \else\hskip5pt\fi\fi
 \hbox{\hskip5pt\vrule width4pt height6pt depth1.5pt\hskip 1 pt}}

\def\cl{\centerline}

\newtheorem{theo}{Theorem}[section]
\newtheorem{lemm}[theo]{Lemma}

\newtheorem{defi}[theo]{Definition}
\newtheorem{coro}[theo]{Corollary}
\newtheorem{prop}[theo]{Proposition}

\begin{document}
\cl{\large\bf Whittaker modules for the derivation Lie algebra }
\cl{{\large\bf of
torus with two variables}
\noindent\footnote{\footnotesize Supported by the National Natural Science Foundation of
China (No. 11271165), the Youth Foundation of National Natural Science
Foundation of China (No. 11101350, 11302052) and the natural Science Foundation of Fujian Province (No. 2010J05001)}} \vspace{16pt}
 \cl{Haifeng  Lian}
 \cl{\small Department of Mathematics, }
 \cl{\small Fujian Agriculture and
Forestry University, Fuzhou 350002,
China}
\cl{Email:
hlian@fafu.edu.cn}
 \vspace{16pt}
 \cl{Xiufu Zhang}
 \cl{\small School of Mathematics and Statistics,}\cl{ \small Jiangsu Normal
University, Xuzhou 221116, China }
\cl{Email: xfzhang@jsnu.edu.cn}

\numberwithin{equation}{section}

\begin{abstract}
Let $\mathcal{L}$ be the derivation
Lie algebra of  ${\mathbb C}[t_1^{\pm 1},t_2^{\pm 1}]$. Given a
triangle decomposition
 $\mathcal{L} =\mathcal{L}^{+}\oplus\mathfrak{h}\oplus\mathcal{L}^{-}$, we define a
nonsingular Lie algebra homomorphism
$\psi:\mathcal{L}^{+}\rightarrow\mathbb{C}$ and the universal
Whittaker $\mathcal{L}$-module $W_{\psi}$ of type $\psi$. We obtain
all  Whittaker vectors and submodules
 of $W_{\psi}$, and  all simple Whittaker $\mathcal{L}$-modules of type
$\psi$.
 \vspace{2mm}\\{\bf 2000 Mathematics Subject
Classification:} 17B10, 17B65, 17B68, 17B81\vspace{2mm}
\\ {\bf Keywords:}  Whittaker vector, Whittaker
module, simple module, Derivation Lie algebra
\end{abstract}

\vskip 3mm\noindent{\section{Introduction}}

Throughout the paper, the symbols ${\mathbb C}$, $\mathbb{N}$,
$\mathbb{Z}$ and $\mathbb{Z}_+$
 represent for the complex field, the set of
non-negative integers, the set of integers, and the set of positive
integers, respectively. Denote
$\underline{m}=\{1,\cdots,m\}$ for $m\geq 1$ and
$\mathbb{Z}^n=\{(a_1,\cdots,a_n)|a_1,\cdots,a_n\in\mathbb{Z}\}$ for
$n\geq 2$.
The rings and Lie algebras in our paper are all over $\mathbb{C}.$

Let $\mathcal{A}_n={\mathbb C}[t_1^{\pm
1},\cdots,t_n^{\pm 1}]$ be the Laurent polynomial ring with
$n$ variables.
Let $\textrm{Der}\mathcal{A}_n$ be the derivation Lie algebra of
$\mathcal{A}_n$. When $n=1,$ $\textrm{Der}\mathcal{A}_1$ is just the Witt algebra.
The central extension of Witt algebra is called the Virasoro algebra, it is widely used in conformal field theory and string theory and
plays important roles, the representation theories of Virasoro algebra are widely studied(see \cite{M,OW, MZ1, MZ2} and the references therein). For $n\geq2,$ $\textrm{Der}\mathcal{A}_n$ is a natural generalization of Witt algebra.
It is  proved that $\textrm{Der}\mathcal{A}_n (n\geq2)$ has no nontrivial central extension. The Virasoro-like Lie algebra is a central extension of a subalgebra of $\textrm{Der}\mathcal{A}_2,$ which has also been widely studied (see \cite{JM, LT2, KP, WZh, ZZ})
The quantum torus contains the Laurent polynomial ring $\mathcal{A}_n$ as its special case.
The derivation Lie algebras of quantum torus
have many important application in the study of the representation theory of Lie algebras (see \cite{LT1} and the references therein).

The purpose of this paper is to study explicitly the Whittaker
modules for
 $\textrm{Der}\mathcal{A}_2$. Our aim is to give the classification of the simple Whittaker modules over $\textrm{Der}\mathcal{A}_2.$

Whittaker modules were first discovered for sl$_2(\mathbb{C})$ by Arnal and
 Pinczon (see \cite{AP}). The versions of Whittaker modules for
 finite-dimensional complex semisimple Lie algebras were
generalized by  Kostant in \cite{Ko}. Since this class of modules
tied up with Whittaker equations in number theory, Kostant called
them Whittaker modules. He proved that for a given
 complex semisimple Lie algebra $\cal G$,
there is a  bijection between its Whittaker modules and the ideals
in the center of the universal enveloping algebra $U(\mathcal{G})$.
The prominent role played by Whittaker modules is illustrated by
Block (see \cite{B}). In his paper, Block  proved that any
irreducible sl$_2(\mathbb{C})$-module belongs to the following three
families of modules: highest (or lowest) weight modules, Whittaker
modules, and
 modules obtained by localization.

Since the definition of  Whittaker modules is closely tied to
triangular decomposition of a finite-dimensional complex semisimple
Lie algebra $\cal G$, it is natural to consider Whittaker modules
for other algebras with a triangular decomposition such as the
generalized Weyl algebras,  Heisenberg algebras, affine Lie
algebras,  Virasoro algebras, Lie algebra of Block type and the
Schr\"{o}inger-Witt algebra, etc.(see \cite{BO, Ch, OW, WZ, ZTL}).

The paper is organized as follows. In Section 2, we construct a
universal Whittaker module for $\textrm{Der}\mathcal{A}_n$. In section 3, we
describe the Whittaker vectors and find all the simple nonsingular Whittaker
modules. In section 4, the simplicity of nonsingular
Whittaker modules is completely determined.

\noindent{\section{Preliminaries}}

Let $\mathcal{L}=\textrm{Der}\mathcal{A}_n$ be the
 derivation Lie algebra of
  $\mathcal{A}_n={\mathbb C}[t_1^{\pm
1},\cdots,t_n^{\pm 1}]$ with $n\geq 2$. For
$\alpha=(\alpha_1,\cdots,\alpha_n)\in {\mathbb{Z}}^{n}$ and
$i\in\underline{n}$, set
 $$d_{i}(\alpha)=t^{\alpha}(t_{i}\frac{\partial}{\partial
 t_{i}}),\quad
  \mathcal{L}_{\alpha}=\oplus_{i\in \underline{n}}\mathbb{C}d_{i}(\alpha),$$  where $t^\alpha=t_1^{\alpha_1}\cdots t_n^{\alpha_n}$.
 Then $\mathcal{L}=\oplus_{\alpha\in
{\mathbb{Z}}^{n}} \mathcal{L}_{\alpha}$ is a
${\mathbb{Z}}^{n}$-graded Lie algebra with the following
multiplication:
$$[d_i(\alpha), d_j(\beta)]= \beta_i d_j(\alpha+\beta)
-\alpha_j d_i(\alpha+\beta).$$

For
$\alpha=(\alpha_1,\cdots,\alpha_n),\beta=(\beta_1,\cdots,\beta_n)\in
\mathbb{Z}^{n}$, we denote $\beta<\alpha$ (or $\alpha>\beta$)  if
the first nonzero elements of
$\alpha_1-\beta_1,\cdots,\alpha_n-\beta_n$ is positive; denote
$\beta\leq\alpha$ (or $\alpha\geq\beta$) if $\beta<\alpha$ or if
$\beta=\alpha$. Clearly, $({\mathbb{Z}}^{n},\leq)$ is a totally
ordered set.

Set ${\mathbb{Z}}^{n}_+=\{ \alpha \in {\mathbb{Z}}^{n}|~
\alpha>\mathbf{0}\}$, ${\mathbb{Z}}^{n}_-=\{ \alpha \in
{\mathbb{Z}}^{n}| \alpha<\mathbf{0}\},$  where
$\mathbf{0}=(0,\cdots,0)\in \mathbb{Z}^{n}$. $\mathcal{L}$ has the
following triangular decomposition
$$\mathcal{L}=\mathcal{L}^{+}\oplus\mathfrak{h}\oplus\mathcal{L}^{-},$$
where $\mathcal{L}^{+}=\oplus_{\alpha\in{\mathbb{Z}}^{n}_+}{\cal
L}_\alpha$ and
$\mathcal{L}^{-}=\oplus_{\alpha\in{\mathbb{Z}}^{n}_-}{\cal
L}_\alpha$ are subalgebras of $\mathcal{L}$, and
$\mathfrak{h}=\mathcal{L}_{\mathbf{0}}$ is a cartan subalgebra of
$\mathcal{L}$.

For
$\alpha=(\alpha_1,\cdots,\alpha_n),\beta=(\beta_1,\cdots,\beta_n)\in
\mathbb{Z}^{n}$, define
$$\alpha\pm\beta=(\alpha_1\pm\beta_1,\cdots,\alpha_n\pm\beta_n).$$ Set
$\varepsilon_i=(\varepsilon_{i1},\cdots,\varepsilon_{in})$ where
$\varepsilon_{ij}=\delta_{ij}$, that is $\varepsilon_{ii}=1$ and
$\varepsilon_{ij}=0$ for $i\neq j$.

\begin{prop} Let $\mathcal{L}^{+}=
\oplus_{i\in\underline{n}}\mathbb{C}d_i(\varepsilon_n)\oplus\mathbb{C}d_n(2\varepsilon_n)\oplus[\mathcal{L}^{+},\mathcal{L}^{+}]$.
\end{prop}
\begin{proof} Since
$\mathcal{L}^{+}=\oplus_{\alpha\in{\mathbb{Z}}^{n}_+}{\cal
L}_\alpha$, we need to show that
$$[\mathcal{L}^{+},\mathcal{L}^{+}]=(\oplus_{i\in\underline{n-1}}\mathbb{C}d_i(2\varepsilon_n))\oplus(\oplus_{\alpha>2\varepsilon_n}{\cal
L}_\alpha).$$ For $\alpha=(\alpha_1,\cdots,\alpha_n)>2\varepsilon_n$
and $i\in \underline{n}$, if $\alpha_n>2$, then we have
$$d_i(\alpha)=\frac{\alpha_i}{(\alpha_n-2)}[d_n(\varepsilon_n),d_n(\alpha-\varepsilon_n)]
-[d_i(\varepsilon_n),d_n(\alpha-\varepsilon_n)]\in
[\mathcal{L}^{+},\mathcal{L}^{+}];$$ if $\alpha_n\leq2$, then there
is $i_0\in \underline{n-1}$ such that $\alpha_{i_0}>0$, thus we have
$$d_i(\alpha)=\frac{1}{\alpha_{i_0}}[d_{i_0}(\varepsilon_n),d_i(\alpha-\varepsilon_n)]
+\frac{1}{\alpha_{i_0}^2}\delta_{in}[d_{i_0}(\varepsilon_n),d_{i_0}(\alpha-\varepsilon_n)]\in
[\mathcal{L}^{+},\mathcal{L}^{+}].$$ Therefore, we have
$\oplus_{\alpha>2\varepsilon_n}{\cal L}_\alpha\subseteq
[\mathcal{L}^{+},\mathcal{L}^{+}]$. Since $[{\cal L}_{\alpha},{\cal
L}_{\beta}]\subseteq {\cal L}_{\alpha+\beta}$ and
$$[\mathcal{L}_{\varepsilon_{n}},
\mathcal{L}_{\varepsilon_n}]=\sum_{i,j\in\underline{n}}
 \mathbb{C}[d_i(\varepsilon_n), d_j(\varepsilon_n)]=\sum_{i\in\underline{n-1}}
 \mathbb{C}[d_n(\varepsilon_n), d_i(\varepsilon_n)]=\oplus_{i\in\underline{n-1}}\mathbb{C}D_i(2\varepsilon_n),$$
 we have  $[\mathcal{L}^{+},\mathcal{L}^{+}]=(\oplus_{i\in\underline{n-1}}d_i(2\varepsilon_n))\oplus(\oplus_{\alpha>2\varepsilon_n}{\cal
L}_\alpha)$, as required.
\end{proof}

Let $\psi:{\cal L}^{+}\rightarrow\mathbb{C}$ be a
homomorphism of Lie algebras. Set $\Omega=\{d_1(\varepsilon_n),
\cdots, d_n(\varepsilon_n)$, $d_n(2\varepsilon_n)\}$. Though
$\mathcal{L}^+$ can not be
 degenerated by $\Omega$, following from proposition 2.1, $\psi$ is
uniquely determined by $\Omega$.

\begin{defi} Suppose $\psi:\mathcal{L}^{+}\rightarrow\mathbb{C}$ is a
homomorphism of Lie algebras. If $\psi(x)\neq 0$ for all
$x\in\Omega$, we say $\psi$ is nonsingular; otherwise, we say
$\psi$ is singular.
\end{defi}

\begin{defi}  Suppose   $V$ is an  $\mathcal{L}$-module and
$\psi:\mathcal{L}^{+}\rightarrow\mathbb{C}$ is a homomorphism of Lie
algebras. A nonzero vector $v$ of $V$ is called  a Whittaker vector
of type $\psi$, if $xv=\psi(x)v$ for every $x\in \mathcal{L}^{+}$.
$V$ is called a  Whittaker module of type $\psi$ if there is a
Whittaker vector $w$ which generates $V$. In this case, we call
 $w$   a cyclic Whittaker vector.  Moreover, if any Whittaker $\mathcal{L}$-module of type $\psi$
is a quotient of $V$, then we call $V$ is universal.
\end{defi}

Suppose $\psi:\mathcal{L}^{+}\rightarrow\mathbb{C}$ is homomorphism of Lie
algebras. For $x\in \mathcal{L}^{+}$, $a\in \mathbb{C}$,  define
$x\cdot a=\psi(x)a$. Then $\mathbb{C}$ is an
$\mathcal{L}^{+}$-module, and is denoted by $\mathbb{C}_\psi$. Let
$W_{\psi}$ be the $\mathcal{L}$-module induced by $\mathbb{C}_\psi$,
that is
$$W_{\psi}=U(\mathcal{L})\otimes_{U(\mathcal{L}^{+})}\mathbb{C}_{\psi}.$$
Clearly,  $W_{\psi}$ is a universal Whittaker module with cyclic
Whittaker vector $w=1\otimes1$, and $W_{\psi}$ is isomorphic to
$U(\mathfrak{b}^{-})$ as $U(\mathfrak{b}^{-})$-module, where
$\mathfrak{b}^{-}=\mathcal{L}^{-}\oplus\mathfrak{h}$.

\noindent{\section{Whittaker vectors for Whittaker modules of nonsingular type}}

In what follows, we let $n=2$,
$\mathcal{L}=\textrm{Der}\mathcal{A}_2$,
$\psi:\mathcal{L}^{+}\rightarrow\mathbb{C}$ be a nonsingular
homomorphism of Lie algebras,
$W_{\psi}=U(\mathcal{L})\otimes_{U(\mathcal{L}^{+})}\mathbb{C}_{\psi}$
be the universal Whittaker $\mathcal{L}$-module of type $\psi$ with
cyclic Whittaker vector $w=1\otimes1$, $\Omega=\{d_1(\varepsilon_2),
d_2(\varepsilon_2)$, $d_2(2\varepsilon_2)\}$.

To describe the bases of the Whittaker module $W_{\psi}$, we need
the following notations.

\begin{defi}  Let $(\Lambda,\leq)$ be a totally ordered set. A partition of $\Lambda$ with length $r$ is a non-decreasing sequence of $r$
  elements of $\Lambda$:
  $$\lambda=(\lambda^1,\cdots,\lambda^r),\quad \lambda^1\leq\cdots\leq\lambda^r.$$
\end{defi}

Denote by $\mathcal{P}(\Lambda)$ the set of all partitions
with
  finite length. For
  $\lambda=(\lambda^1,\cdots,\lambda^r)\in\mathcal{P}(\Lambda)$,
  denote $l(\lambda)=r$, and for $\alpha\in\Lambda$, let
  $\lambda(\alpha)$ denote the number of  times $\alpha$ appears in
  the partition. Clearly, any partition $\lambda$ is completely
  determined by the values $\lambda(\alpha)$, $\alpha\in\Lambda$.
  If all $\lambda(\alpha)=0$, call $\lambda$ the null partition and
  denote $\lambda=\bar{0}$.
  Note that, $\bar{0}$ is the only partition of with length $0$. We
  consider $\bar{0}$ an element of $\mathcal{P}(\Lambda)$. For $\lambda\in\mathcal{P}(\Lambda)$, we
also write $\lambda=\{
\alpha^{\lambda(\alpha)}\}_{\alpha\in{\Lambda}}$.

Let $\mathcal{P}=\mathcal{P}({\mathbb{Z}}^{n}_+)$,
$\widetilde{\mathcal{P}}=\mathcal{P}\times \mathcal{P}\times
\mathbb{N}$. For
$\lambda=(\lambda^1,\cdots,\lambda^r)\in\mathcal{P}$ and $i=1,2$,
set
$$\begin{array}{l} \vspace{1mm} |\lambda|=|\lambda^1|+\cdots+|\lambda^r|,\quad|\bar{0}|=\mathbf{0},\\
 \vspace{1mm}x_{i,\lambda}=d_i(-\lambda^1)\cdots
d_i(-\lambda^r),\quad x_{i,\bar{0}}=1.
\end{array}$$
For $(\lambda,\mu,k)\in\widetilde{\mathcal{P}}$,  $
\alpha\in{\mathbb{Z}}^{2}_+ $, we set
   $$\begin{array}{l} \vspace{1mm} |(\lambda,\mu,k)|=
   |\lambda|+|\mu|,\quad\quad\quad  x_{\lambda,\mu,k}=
   x_{1,\lambda}x_{2,\mu}d_2^k(\mathbf{0}),\\ \vspace{1mm}
S_\mu=\{ \gamma\in{\mathbb{Z}}^{2}_+ | \mu(\gamma)> 0 \},\quad\quad
S_{\lambda,\mu}=\{
\gamma\in{\mathbb{Z}}^{2}_+\ |\ \lambda(\gamma)\neq\mu(\gamma)\},\\
\vspace{1mm} S^+_\mu=\{
\gamma=(\gamma_1,\gamma_2)\in{\mathbb{Z}}^{2}_+ | \mu(\gamma)> 0\
\textrm{and}\ \gamma_1>0 \},\\
\vspace{1mm}S^+_{\lambda,\mu}=\{
\gamma=(\gamma_1,\gamma_2)\in{\mathbb{Z}}^{2}_+\ |
\lambda(\gamma)\neq\mu(\gamma)\ \textrm{and} \ \gamma_1>0\},\\
\mu_\alpha=\{ \gamma^{k_\gamma}\}_{\gamma\in{\mathbb{Z}}^{2}_+},\
\textrm{where} \ k_\alpha=\mu(\alpha)-1 \textrm{ and }
   k_\gamma=\mu(\gamma) \textrm{ for } \gamma\neq\alpha.\end{array}$$
 We denote $x_{i,\mu_\alpha}=0$, if $\mu_\alpha\not\in\mathcal{P}$.

Set $z=d_{1}(\mathbf{0})$. Following from the
Poincar\'{e}-Birkhoff-Witt theorem, the set $$ \{
x_{\lambda,\mu,k}z^rw\ |\ (\lambda,\mu,k)\in
\widetilde{\mathcal{P}}, \ r\in \mathbb{N} \}$$ forms a basis of
$W_{\psi}$.

\begin{theo} Suppose $\psi:\mathcal{L}^{+}\rightarrow\mathbb{C}$ is
 a nonsingular
homomorphism of Lie algebras. Let  $W_{\psi}$ be a universal
Whittaker $\mathcal{L}$-module with cyclic Whittaker vector  $w$.
For $0\neq u\in U(\mathfrak{b}^{-})$,  $uw$ is a  Whittaker vector
if and only if  $u\in \mathbb{C}[z]$.
\end{theo}

In  the proof, we need the following notations and lemmas.

\begin{defi} Suppose
$\lambda,\mu\in\mathcal{P}$, $\lambda\neq\mu$.

(1) We say $\lambda<\mu$ if $\lambda(\alpha)<\mu(\alpha)$,
where $\alpha$ is the minimum element in $S_{\lambda,\mu}$.

(2)  We say $\lambda\prec\mu$ if
$S^+_{\lambda,\mu}\neq\emptyset$ and $\lambda(\alpha)<\mu(\alpha)$,
where $\alpha$ is the minimum element in $S^+_{\lambda,\mu}$, or if
$S^+_{\lambda,\mu}=\emptyset$ and $\lambda<\mu$.
\end{defi}

\begin{defi} Suppose
Suppose
$(\lambda,\mu,k),(\lambda',\mu',k')\in\widetilde{\mathcal{P}}$. We
say $(\lambda,\mu,k)\prec(\lambda',\mu',k')$ if it satisfies one of
the following conditions.

(1) $|\lambda|+|\mu|<|\lambda'|+|\mu'|$;

(2) $|\lambda|+|\mu|=|\lambda'|+|\mu'|$, $k<k'$;

(3) $|\lambda|+|\mu|=|\lambda'|+|\mu'|$, $k=k'$, $\mu<\mu'$;

(4) $|\lambda|+|\mu|=|\lambda'|+|\mu'|$, $k=k'$, $\mu=\mu'$,
$\lambda\prec\lambda'$.
\end{defi}

We denote
$(\lambda,\mu,k)\preceq(\lambda',\mu',k')$ if
$(\lambda,\mu,k)=(\lambda',\mu',k')$ or
$(\lambda,\mu,k)\prec(\lambda',\mu',k')$. Then
$(\widetilde{\mathcal{P}},\preceq)$ is a totally ordered set with
minimum element $(\bar{0},\bar{0},0)$.

Set $W_{\psi}(\bar{0},\bar{0},0)=\{0\}$, and for
$(\bar{0},\bar{0},0)\prec
 (\lambda',\mu',k')$, set
$$W_{\psi}(\lambda',\mu',k')=\sum_{(\lambda,\mu,k)\prec(\lambda',\mu',k')}
 x_{\lambda,\mu,k}\mathbb{C}[z]w.$$
 For $0\neq v= \sum_{(\lambda,\mu,k)\in
\widetilde{\mathcal{P}}}
 x_{\lambda,\mu,k}f_{\lambda,\mu,k}(z)w$, where
$f_{\lambda,\mu,k}(z)\in \mathbb{C}[z]$, set
$$P(v)=\{(\lambda,\mu,k)|
f_{\lambda,\mu,k}(z)\neq0\},$$
$$\textrm{deg}(v)=\max\! _\preceq\{(\lambda,\mu,k)|(\lambda,\mu,k)\in
P(v)\}.$$

\begin{lemm} Suppose
$(\lambda,\mu,k)\in \widetilde{\mathcal{P}}$, $0\neq
f(z)\in\mathbb{C}[z]$, $D\in \Omega$. We have
$$D\cdot x_{\lambda,\mu,k}f(z)w \equiv  \psi(D)x_{\lambda,\mu,k}f(z)w  \quad (\textrm{mod}\ W_{\psi}(\lambda,\mu,k))\eqno(3.1)$$
Moreover, if $(\lambda',\mu',k')\in \widetilde{\mathcal{P}}$
and $(\lambda,\mu,k)\prec(\lambda',\mu',k')$, then
$$D\cdot x_{\lambda,\mu,k}f(z)w \equiv  0  \quad (\textrm{mod}\ W_{\psi}(\lambda',\mu',k'))\eqno(3.2)$$
\end{lemm}
\begin{proof} For $(3.1)$, using $Dz=zD$, we have
$$\begin{array}{rl}\vspace{1mm}
&D\cdot x_{\lambda,\mu,k}f(z)w\\ \vspace{2mm}
\equiv &[D,x_{1,\lambda}x_{2,\mu}]d_2^k(\mathbf{0})f(z)w+x_{1,\lambda}x_{2,\mu}Dd_2^k(\mathbf{0})f(z)w\\
\equiv&x_{1,\lambda}x_{2,\mu}[D,d_2^k(\mathbf{0})]f(z)w
+x_{1,\lambda}x_{2,\mu}d_2^k(\mathbf{0})Df(z)w\\
\equiv&x_{1,\lambda}x_{2,\mu}d_2^k(\mathbf{0})f(z)Dw\\
\equiv& \psi(D)x_{1,\lambda}x_{2,\mu}d_2^k(\mathbf{0})f(z)w \quad
(\textrm{mod}\ W_{\psi}(\lambda,\mu,k))
\end{array},$$
as required. Following from $(3.1)$, we have $(3.2)$.
\end{proof}

\begin{lemm} Any Whittaker vector
of $W_{\psi}$ is of type $\psi$.
\end{lemm}
\begin{proof} Suppose $w'= \sum_{(\lambda,\mu,k)\in
\widetilde{\mathcal{P}}}
 x_{\lambda,\mu,k}f_{\lambda,\mu,k}(z)w$ is a  Whittaker vector
of $W_{\psi}$ of type $\psi'$, where
$\psi':\mathcal{L}^{+}\rightarrow\mathbb{C}$ is a Lie homomorphism,
  $f_{\lambda,\mu,k}(z)\in
\mathbb{C}[z]$. Let $(\lambda',\mu',k')=\deg (w')$. Using lemma 3.5,
for $D\in \Omega$, we have
$$\psi'(D)w'=Dw'\equiv \psi(D)w'\quad (\textrm{mod}\
W_{\psi}(\lambda',\mu',k'))$$ which follows $\psi^{'}(D)=\psi(D)$.
Therefore, we have $\psi'=\psi$.
\end{proof}

Suppose $0\neq u\in
  U(\mathfrak{b}^{-})$.  Following from the
Poincar\'{e}-Birkhoff-Witt theorem, $uw$ can be write as
$$uw=\sum_{(\lambda,\mu,k)\in \widetilde{\mathcal{P}}}
x_{\lambda,\mu,k}f_{\lambda,\mu,k}(z)w,\ \textrm{where }
f_{\lambda,\mu,k}(z)\in \mathbb{C}[z].$$ Let
$(\lambda',\mu',k')=\deg(uw)$, $N=|\lambda'|+|\mu'|$,
$$\Lambda_N:=\{ (\lambda,\mu,k)\in P(uw) |\
  |\lambda|+|\mu|=N \},$$
 $$\Lambda_{N-\varepsilon}:=\{  (\lambda,\mu,k)\in P(uw) |
  |\lambda|+|\mu|=N-\varepsilon \}.$$

\begin{lemm} Suppose $k'>0$, we have
$$\begin{array}{ll}\vspace{2mm} & d_2(2\varepsilon_2)\cdot uw-\psi(d_2(2\varepsilon_2))uw\\ \equiv &
-2k'\psi(d_2(2\varepsilon_2))x_{\lambda',\mu',k'-1}f_{\lambda',\mu',k'}(z)w
 \quad (\textrm{mod}\ W_{\psi}(\lambda',\mu',k'-1))\end{array}$$
\end{lemm}
\begin{proof} Since
$d_2(2\varepsilon_2)z=zd_2(2\varepsilon_2)$ and
$d_2(2\varepsilon_2)d_2^k(\mathbf{0})=
(d_2(\mathbf{0})-2)^{k}d_2(2\varepsilon_2)$, we have
$$\begin{array}{rl}\vspace{2mm}
& d_2(2\varepsilon_2)\cdot
uw-\psi(d_2(2\varepsilon_2))uw\\
\vspace{2mm} \equiv&
\sum_{(\lambda,\mu,k)\in\Lambda_N}[d_2(2\varepsilon_2),
x_{\lambda,\mu,k}]f_{\lambda,\mu,k}(z)w\\
\vspace{2mm}
\equiv&\sum_{(\lambda,\mu,k)\in\Lambda_N}[d_2(2\varepsilon_2),
x_{1,\lambda}x_{2,\mu}d_2^k(\mathbf{0})]f_{\lambda,\mu,k}(z)w\\
\vspace{2mm} \equiv&\sum_{(\lambda,\mu,k)\in\Lambda_N}
x_{1,\lambda}x_{2,\mu}[d_2(2\varepsilon_2),d_2^k(\mathbf{0})]f_{\lambda,\mu,k}(z)w
\\
\vspace{2mm} \equiv&\sum_{(\lambda,\mu,k)\in\Lambda_N}
x_{1,\lambda}x_{2,\mu}((d_2(\mathbf{0})-2)^{k}-d_2^k(\mathbf{0}))d_2(2\varepsilon_2)f_{\lambda,\mu,k}(z)w
\\
\vspace{2mm} \equiv&\sum_{(\lambda,\mu,k)\in\Lambda_N}
(-2k)x_{1,\lambda}x_{2,\mu}d_2^{k-1}(\mathbf{0})f_{\lambda,\mu,k}(z)d_2(2\varepsilon_2)w
\\ \vspace{2mm}\equiv&
-2k'\psi(d_2(2\varepsilon_2))x_{1,\lambda'}x_{2,\mu'}d_2^{k'-1}(\mathbf{0})f_{\lambda',\mu',k'}(z)w
\\
\vspace{2mm}\equiv&
-2k'\psi(d_2(2\varepsilon_2))x_{\lambda',\mu',k'-1}f_{\lambda',\mu',k'}(z)w
\quad (\textrm{mod}\ W_{\psi}(\lambda',\mu',k'-1)).
\end{array}$$
\end{proof}

\begin{lemm} Suppose $k'=0$ and $\mu'\neq\bar{0}$.
Let $\alpha=(\alpha_1,\alpha_2)$ be the minimum element in
$S_{\mu'}$.

(1) If $\alpha_1>0$, then
 $$\begin{array}{ll}&(d_1(\alpha+2\varepsilon_2)-\psi(d_1(\alpha+2\varepsilon_2)))
uw\\
\equiv&
-\alpha_1\mu'(\alpha)\psi(d_2(2\varepsilon_2))x_{1,\lambda'}x_{2,\mu_\alpha'}f_{\lambda',\mu',0}(z)w
  \quad(\textrm{mod}\ W_\psi(\lambda',\mu_\alpha',0))
\end{array}$$

(2) If $\alpha_1=0$, then
 $$\begin{array}{ll}&(d_2(\alpha+2\varepsilon_2)-\psi(d_2(\alpha+2\varepsilon_2)))
uw\\
\equiv&
-2(\alpha_2+1)\mu'(\alpha)\psi(d_2(2\varepsilon_2))x_{1,\lambda'}x_{2,\mu_\alpha'}f_{\lambda',\mu',0}(z)w
  \quad(\textrm{mod}\ W_\psi(\lambda',\mu_\alpha',0))
\end{array}$$
\end{lemm}
\begin{proof} By the definitions of  $(\lambda',\mu',k'), N, \Lambda_N$, we have
 $$uw=\!\!\sum_{(\lambda,\mu,k)\not\in\Lambda_N}\!\!
x_{\lambda,\mu,k}f_{\lambda,\mu,k}(z)w
+\!\!\sum_{\substack{(\lambda,\mu,k)\prec (\lambda',\mu',k')\\
(\lambda,\mu,k)\in\Lambda_N}}\!\!
x_{\lambda,\mu,k}f_{\lambda,\mu,k}(z)w
+x_{\lambda',\mu',k'}f_{\lambda',\mu',k'}(z)w.$$

For (1), let $(\lambda,\mu,k)\in P(uw)$ and
$v=d_1(\alpha+2\varepsilon_2)
x_{\lambda,\mu,k}f_{\lambda,\mu,k}(z)w$. If
$(\lambda,\mu,k)\not\in\Lambda_N$, then $|\lambda|+|\mu|<N$, which
implies $|\deg v|\leq|\lambda|+|\mu|-\alpha<N-\alpha$ while $v\neq
0$. Thus we have
$$d_1(\alpha+2\varepsilon_2)
x_{\lambda,\mu,k}f_{\lambda,\mu,k}(z)w\equiv 0
  \quad(\textrm{mod}\ W_\psi(\lambda',\mu_\alpha',0)).$$
If $(\lambda,\mu,k)\in\Lambda_N$ but $(\lambda,\mu,k)\prec
(\lambda',\mu',0)$, we have $k=0$ and
$$\begin{array}{ll}\vspace{2mm} & d_1(\alpha+2\varepsilon_2)
x_{\lambda,\mu,0}f_{\lambda,\mu,0}(z)w\\ \vspace{2mm} \equiv &
 d_1(\alpha+2\varepsilon_2)
x_{1,\lambda}x_{2,\mu}f_{\lambda,\mu,0}(z)w  \\
\vspace{2mm} \equiv & [d_1(\alpha+2\varepsilon_2),
x_{1,\lambda}]x_{2,\mu}f_{\lambda,\mu,0}(z)w +
x_{1,\lambda}d_1(\alpha+2\varepsilon_2)x_{2,\mu}f_{\lambda,\mu,0}(z)w
\\
\vspace{2mm} \equiv &
x_{1,\lambda}[d_1(\alpha+2\varepsilon_2),x_{2,\mu}]f_{\lambda,\mu,0}(z)w
+
x_{1,\lambda}x_{2,\mu}d_1(\alpha+2\varepsilon_2)f_{\lambda,\mu,0}(z)w
\\
\vspace{2mm} \equiv &
-\alpha_1\mu(\alpha)\psi(d_2(2\varepsilon_2))x_{1,\lambda}x_{2,\mu_\alpha}f_{\lambda,\mu,0}(z)w\\
\equiv &0
  \quad(\textrm{mod}\ W_\psi(\lambda',\mu_\alpha',0)).\end{array}$$
Therefore, we have
$$\begin{array}{ll}\vspace{2mm} & (d_1(\alpha+2\varepsilon_2)-\psi(d_1(\alpha+2\varepsilon_2)))
uw\\ \vspace{2mm} \equiv &
 d_1(\alpha+2\varepsilon_2)
x_{1,\lambda'}x_{2,\mu'}f_{\lambda',\mu',0}(z)w  \\
\vspace{2mm} \equiv & [d_1(\alpha+2\varepsilon_2),
x_{1,\lambda'}]x_{2,\mu'}f_{\lambda',\mu',0}(z)w +
x_{1,\lambda'}d_1(\alpha+2\varepsilon_2)x_{2,\mu'}f_{\lambda',\mu',0}(z)w
\\
\vspace{2mm} \equiv &
x_{1,\lambda'}[d_1(\alpha+2\varepsilon_2),x_{2,\mu'}]f_{\lambda',\mu',0}(z)w
+
x_{1,\lambda'}x_{2,\mu'}d_1(\alpha+2\varepsilon_2)f_{\lambda',\mu',0}(z)w
\\
 \equiv &
\alpha_1\mu'(\alpha)\psi(d_2(2\varepsilon_2))x_{1,\lambda'}x_{2,\mu_\alpha'}f_{\lambda',\mu',0}(z)w
  \quad(\textrm{mod}\ W_\psi(\lambda',\mu_\alpha',0)).\end{array}$$

For (2), let $(\lambda,\mu,k)\in P(uw)$ and
$v=d_2(\alpha+2\varepsilon_2)
x_{\lambda,\mu,k}f_{\lambda,\mu,k}(z)w$. If
$(\lambda,\mu,k)\not\in\Lambda_N$, then $|\lambda|+|\mu|<N$, which
implies $|\deg v|\leq|\lambda|+|\mu|-\alpha<N-\alpha$ while $v\neq
0$. Thus we have
$$d_2(\alpha+2\varepsilon_2)
x_{\lambda,\mu,k}f_{\lambda,\mu,k}(z)w\equiv 0
  \quad(\textrm{mod}\ W_\psi(\lambda',\mu_\alpha',0)).$$
If $(\lambda,\mu,k)\in\Lambda_N$ but $(\lambda,\mu,k)\prec
(\lambda',\mu',0)$, we have $k=0$ and $\mu\leq \mu'$. Using
$\alpha_1=0$ and $\alpha_2>0$, we have
$$\begin{array}{ll}\vspace{2mm} & d_2(\alpha+2\varepsilon_2)
x_{\lambda,\mu,k}f_{\lambda,\mu,0}(z)w\\ \vspace{2mm} \equiv &
 d_2(\alpha+2\varepsilon_2)
x_{1,\lambda}x_{2,\mu}f_{\lambda,\mu,0}(z)w  \\
\vspace{2mm} \equiv & [d_2(\alpha+2\varepsilon_2),
x_{1,\lambda}]x_{2,\mu}f_{\lambda,\mu,0}(z)w +
x_{1,\lambda}d_2(\alpha+2\varepsilon_2)x_{2,\mu}f_{\lambda,\mu,0}(z)w
\\
\vspace{2mm} \equiv &
x_{1,\lambda}[d_2(\alpha+2\varepsilon_2),x_{2,\mu}]f_{\lambda,\mu,0}(z)w
+
x_{1,\lambda}x_{2,\mu}d_2(\alpha+2\varepsilon_2)f_{\lambda,\mu,0}(z)w
\\
\vspace{2mm} \equiv &
-2(\alpha_2+1)\mu(\alpha)\psi(d_2(2\varepsilon_2))x_{1,\lambda}x_{2,\mu_\alpha}f_{\lambda,\mu,0}(z)w\\
\equiv &0
  \quad(\textrm{mod}\ W_\psi(\lambda',\mu_\alpha',0)).\end{array}$$
Therefore, we have
$$\begin{array}{ll}\vspace{2mm} & (d_2(\alpha+2\varepsilon_2)-\psi(d_2(\alpha+2\varepsilon_2)))
uw\\ \vspace{2mm} \equiv &
 d_2(\alpha+2\varepsilon_2)
x_{1,\lambda'}x_{2,\mu'}f_{\lambda',\mu',0}(z)w  \\
\vspace{2mm} \equiv & [d_2(\alpha+2\varepsilon_2),
x_{1,\lambda'}]x_{2,\mu'}f_{\lambda',\mu',0}(z)w +
x_{1,\lambda'}d_2(\alpha+2\varepsilon_2)x_{2,\mu'}f_{\lambda',\mu',0}(z)w
\\
\vspace{2mm} \equiv &
x_{1,\lambda'}[d_2(\alpha+2\varepsilon_2),x_{2,\mu'}]f_{\lambda',\mu',0}(z)w
+
x_{1,\lambda'}x_{2,\mu'}d_2(\alpha+2\varepsilon_2)f_{\lambda',\mu',0}(z)w
\\
 \equiv &
-2(\alpha_2+1)\mu'(\alpha)\psi(d_2(2\varepsilon_2))x_{1,\lambda'}x_{2,\mu_\alpha'}f_{\lambda',\mu',0}(z)w
  \quad(\textrm{mod}\ W_\psi(\lambda',\mu_\alpha',0)).\end{array}$$
\end{proof}

\begin{lemm} Suppose $\mu'=\bar{0}$, $k'=0$, $S_{\lambda'}^+\neq\emptyset$.
  Let
  $\alpha=(\alpha_1,\alpha_2)$ be the minimum element in
  $S_{\lambda'}^+$. We have$$\begin{array}{ll}&(d_2(\alpha+2\varepsilon_2)-\psi(d_2(\alpha+2\varepsilon_2)))
uw\\
\equiv& -\alpha_1\lambda'(\alpha)\psi(d_2(2\varepsilon_2))
 x_{1,\lambda_\alpha'}f_{\lambda',\bar{0},0}(z)w
  \quad(\textrm{mod}\ W_\psi(\lambda_\alpha',\bar{0},0))
\end{array}$$
\end{lemm}
\begin{proof} Let $(\lambda,\mu,k)\in P(uw)$ and
$v=d_2(\alpha+2\varepsilon_2)
x_{\lambda,\mu,k}f_{\lambda,\mu,k}(z)w$.
 If $|(\lambda,\mu,k)|<|\lambda'|$, then  $|\deg
v|< |\lambda'|-\alpha$ while $v\neq 0$. Thus we have
$$d_2(\alpha+2\varepsilon_2)
x_{\lambda,\mu,k}f_{\lambda,\mu,k}(z)w\equiv 0
  \quad(\textrm{mod}\ W_\psi(\lambda_\alpha',\bar{0},0)).$$

If $|(\lambda,\mu,k)|=|\lambda'|$, we have $\mu=\bar{0}$, $k=0$ and
$\lambda\preceq\lambda'$. We may write
$$d_{1,\lambda}=(d_1(-\alpha))^{\lambda(\alpha)}d_1(-\beta^{(1)})\cdots
d_1(-\beta^{(m)})d_1(-\gamma^{(1)})\cdots d_1(-\gamma^{(l)})$$ where
$\alpha<\beta^{(1)}\leq\cdots\leq\beta^{(m)}$,
$\textbf{0}<\gamma^{(1)}\leq\cdots \leq\gamma^{(l)}$ and
$\gamma^{(l)}=(0,\gamma^{(l)}_2)$. Using $\alpha_1>0$, we have
$$\begin{array}{ll}\vspace{2mm} & d_2(\alpha+2\varepsilon_2)
x_{\lambda,\bar{0},0}f_{\lambda,\bar{0},0}(z)w\\ \vspace{2mm} \equiv
&
 d_2(\alpha+2\varepsilon_2)
(d_1(-\alpha))^{\lambda(\alpha)}d_1(-\beta^1)\cdots
d_1(-\beta^m)d_1(-\gamma^1)\cdots d_1(-\gamma^l)f_{\lambda,\bar{0},0}(z)w \\
\vspace{2mm} \equiv & [d_2(\alpha+2\varepsilon_2),
(d_1(-\alpha))^{\lambda(\alpha)}]d_1(-\beta^{(1)})\cdots
d_1(-\beta^{(m)})d_1(-\gamma^{(1)})\cdots
d_1(-\gamma^{(l)})f_{\lambda,\bar{0},0}(z)w \\ \vspace{2mm} &+
(d_1(-\alpha))^{\lambda(\alpha)}d_2(\alpha+2\varepsilon_2)d_1(-\beta^{(1)})\cdots
d_1(-\beta^{(m)})d_1(-\gamma^{(1)})\cdots
d_1(-\gamma^{(l)})f_{\lambda,\bar{0},0}(z)w
\\
\vspace{2mm} \equiv &
-\alpha_1\lambda(\alpha)(d_1(-\alpha))^{\lambda(\alpha)-1}d_2(2\varepsilon_2)d_1(-\beta^{(1)})\cdots
d_1(-\beta^{(m)})d_1(-\gamma^{(1)})\cdots
d_1(-\gamma^{(l)})f_{\lambda,\bar{0},0}(z)w \\ \vspace{2mm} \equiv
&-\alpha_1\lambda(\alpha)\psi(d_2(2\varepsilon_2))d_{1,\lambda_\alpha}f_{\lambda,\bar{0},0}(z)w
\\ \equiv
&-\delta_{\lambda,\lambda'}\alpha_1\lambda'(\alpha)\psi(d_2(2\varepsilon_2))d_{1,\lambda_\alpha'}f_{\lambda',\bar{0},0}(z)w
  \quad(\textrm{mod}\ W_\psi(\lambda_\alpha',\bar{0},0))\end{array}$$
where $\delta_{\lambda,\lambda'}=1$ if $\lambda=\lambda'$, and
$\delta_{\lambda,\lambda'}=0$ if $\lambda\neq\lambda'$. Therefore,
we have
$$
\begin{array}{ll}
&(d_2(\alpha+2\varepsilon_2)-\psi(d_2(\alpha+2\varepsilon_2)))uw\\
\vspace{2mm} \equiv& d_2(\alpha+2\varepsilon_2)uw\\
\vspace{2mm}\equiv&-\alpha_1\lambda'(\alpha)\psi(d_2(2\varepsilon_2))d_{1,\lambda_\alpha'}f_{\lambda',\bar{0},0}(z)w
  \quad(\textrm{mod}\ W_\psi(\lambda_\alpha',\bar{0},0))
\end{array}$$
as required.
\end{proof}

\begin{lemm} Suppose
$\mu'=\bar{0}$, $k'=0$, $S_{\lambda'}^+=\emptyset$,
$S_{\lambda'}\neq\emptyset$ and $\Lambda_{N-\varepsilon_2}=
\emptyset$. Let
  $\alpha=(0,\alpha_2)$ be the minimum element in $S_{\lambda'}$, we have
  $$\begin{array}{ll}&(d_2(\alpha+\varepsilon_2)-\psi(d_2(\alpha+\varepsilon_2)))
uw\\
\equiv& -\alpha_2\lambda'(\alpha)\psi(d_1(\varepsilon_2))
 x_{\lambda_\alpha',\bar{0},0}f_{\lambda',\bar{0},0}(z)w
  \quad(\textrm{mod}\ W_\psi(\lambda_\alpha',\bar{0},0))
\end{array}$$
\end{lemm}
\begin{proof} Since $\Lambda_{N-\varepsilon_2}= \emptyset$,
we can write
$$uw=\sum_{|(\lambda,\mu,k)|\leq N-2\varepsilon_2}x_{\lambda,\mu,k}f_{\lambda,\mu,k}(z)
+\sum_{(\lambda,\bar{0},0)\in
\Lambda_N}x_{1,\lambda}f_{\lambda,\bar{0},0}(z).$$ By the
definitions of $\psi$ and $W_\psi(\lambda_\alpha',\bar{0},0)$, we
have
$$
\begin{array}{ll} &(d_2(\alpha+\varepsilon_2)-\psi(d_2(\alpha+\varepsilon_2)))uw\\
\vspace{2mm} \equiv & \sum_{(\lambda,\bar{0},0)\in
\Lambda_N}[d_2(\alpha+\varepsilon_2),x_{1,\lambda}]f_{\lambda,\bar{0},0}(z)
\\
\vspace{2mm} \equiv &\sum_{(\lambda,\bar{0},0)\in
\Lambda_N}(-\alpha_2)\lambda(\alpha)\psi(d_1(\varepsilon_2))
x_{1,\lambda_\alpha}f_{\lambda,\bar{0},0}(z)w\\
\vspace{2mm}\equiv&
-\alpha_2\lambda'(\alpha)\psi(d_1(\varepsilon_2))
 x_{\lambda_\alpha',\bar{0},0}f_{\lambda',\bar{0},0}(z)w
  \quad(\textrm{mod}\ W_\psi(\lambda_\alpha',\bar{0},0))
\end{array}$$
as required.
\end{proof}

\begin{lemm} Suppose
$\mu'=\bar{0}$, $k'=0$, $S_{\lambda'}^+=\emptyset$,
$S_{\lambda'}\neq\emptyset$ and $\Lambda_{N-\varepsilon_2}\neq
\emptyset$. Let
 $(\xi,\eta,l)$ be the maximum element in $\Lambda_{N-\varepsilon_2}$.

(1) If $l> 0$, we have
$$
\begin{array}{ll} &(d_1(\varepsilon_2)-\psi(d_1(\varepsilon_2)))uw\\
\vspace{2mm} \equiv & -l\psi(d_1(\varepsilon_2))
x_{\xi,\eta,l-1}f_{\xi,\eta,l}(z)w \quad(\textrm{mod}\
W_\psi(\xi,\eta,l-1))
\end{array}$$

(2) If  $l=0$ but $\eta\neq\bar{0}$, then we have
$$
\begin{array}{ll} &(d_1(\beta+\varepsilon_2)-\psi(d_1(\beta+\varepsilon_2)))uw\\
\vspace{2mm} \equiv & -(\beta+1)\eta(\beta)\psi(d_1(\varepsilon_2))
x_{\xi,\eta_\beta,0}f_{\xi,\eta,0}(z)w \quad(\textrm{mod}\
W_\psi(\xi,\eta_\beta,0))
\end{array}$$
where $\beta=(0,\beta_2)$ is the minimum element in $S_{\eta_2}$;

(3) If $l=0$ and $\eta=\bar{0}$, then we have
$$\begin{array}{ll} &(d_2(\alpha+\varepsilon_2)-\psi(d_2(\alpha+\varepsilon_2)))uw\\
 \vspace{2mm} \equiv
&-\alpha_2\lambda'(\alpha)\psi(d_2(\varepsilon_2))
x_{\lambda_\alpha',\bar{0},0}f_{\lambda',\bar{0},0}(z)w
\quad(\textrm{mod}\ W_\psi(\lambda_\alpha',\bar{0},0))
\end{array}$$
where $\alpha=(0,\alpha_2)$ be the minimum element in
$S_{\lambda'}$.
\end{lemm}
\begin{proof} Since $\mu'=\bar{0}$, $k'=0$,
$S_{\lambda'}^+=\emptyset$ and $S_{\lambda'}\neq\emptyset$, there is
a positive integer $n$ such that $N=(0,n)$. For $(\lambda,\mu,k)\in
\Lambda_N$, we have $\mu=\bar{0}$, $k=0$ and $|\lambda|=(0,n)$.
Thus, we have $[d_1(a\varepsilon_2),x_{\lambda,\mu,k}]=0$, which
implies
$(d_1(a\varepsilon_2)-\psi(d_1(a\varepsilon_2)))x_{\lambda,\mu,k}f_{\lambda,\mu,k}(z)w=0$,
where $a$ is an integer.

For (1) and (2), we can write
$$uw=\!\!\!\!\!\!\!\!\sum_{|(\lambda,\mu,k)|\leq
N-2\varepsilon_2}\!\!\!\!\!\!\!\!x_{\lambda,\mu,k}f_{\lambda,\mu,k}(z)w
+\!\!\!\!\!\!\!\!\sum_{\substack{(\lambda,\mu,k)\in
\Lambda_{N-\varepsilon_2}\\
(\lambda,\mu,k)\prec(\xi,\eta,l)}}\!\!\!\!\!\!\!\!x_{\lambda,\mu,k}f_{\lambda,\mu,k}(z)w
+\!\!\!\!\!\!\sum_{(\lambda,\bar{0},0)\in
\Lambda_N}\!\!\!\!\!\!x_{1,\lambda}f_{\lambda,\bar{0},0}(z)w+x_{\xi,\eta,l}f_{\xi,\eta,l}(z)w.$$
If $l>0$, then using the definitions of $\psi$ and
$W_\psi(\xi,\eta,l-1)$, we have
$$
\begin{array}{ll} &(d_1(\varepsilon_2)-\psi(d_1(\varepsilon_2)))uw\\
\vspace{2mm} \equiv & \sum_{\substack{(\lambda,\mu,k)\in
\Lambda_{N-\varepsilon_2}\\
(\lambda,\mu,k)\prec(\xi,\eta,l)}}[d_1(\varepsilon_2),x_{\lambda,\mu,k}f_{\lambda,\mu,k}(z)]w
+[d_1(\varepsilon_2),x_{\xi,\eta,l}f_{\xi,\eta,l}(z)]w\\
\vspace{2mm} \equiv &
x_{1,\xi}x_{2,\eta}[d_1(\varepsilon_2),d_2^l(\mathbf{0})]f_{\xi,\eta,l}(z)w\\
\vspace{2mm} \equiv & -l\psi(d_1(\varepsilon_2))
x_{1,\xi}x_{2,\eta}d_2^{l-1}(\mathbf{0})f_{\xi,\eta,l}(z)w\\
\vspace{2mm} \equiv &-l\psi(d_1(\varepsilon_2))
x_{\xi,\eta,l-1}f_{\xi,\eta,l}(z)w \quad(\textrm{mod}\
W_\psi(\xi,\eta,l-1))
\end{array}$$
as required.  If $l=0$ but $\eta\neq\bar{0}$, then let $\beta$ be
the minimum element in $S_{\eta}$. Clearly, $\beta=(0,\beta_2)$ for
some $\beta_2\geq0$. Using the definitions of $\psi$ and
$W_\psi(\xi,\eta_\beta,0)$, we have
$$
\begin{array}{ll} &(d_1(\beta+\varepsilon_2)-\psi(d_1(\beta+\varepsilon_2)))uw\\
\vspace{2mm} \equiv & \sum_{\substack{(\lambda,\mu,0)\in
\Lambda_{N-\varepsilon_2}\\
(\lambda,\mu,0)\prec(\xi,\eta,0)}}[d_1(\beta+\varepsilon_2),x_{\lambda,\mu,0}f_{\lambda,\mu,0}(z)]w
+[d_1(\beta+\varepsilon_2),x_{\xi,\eta,0}f_{\xi,\eta,0}(z)]w\\
\vspace{2mm} \equiv &
x_{1,\xi}[d_1(\beta+\varepsilon_2),x_{2,\eta}]f_{\xi,\eta,0}(z)w\\
\vspace{2mm} \equiv & -(\beta+1)\eta(\beta)\psi(d_1(\varepsilon_2))
x_{1,\xi}x_{2,\eta_\beta}f_{\xi,\eta,0}(z)w\\
\vspace{2mm} \equiv &-(\beta+1)\eta(\beta)\psi(d_1(\varepsilon_2))
x_{\xi,\eta_\beta,0}f_{\xi,\eta,0}(z)w \quad(\textrm{mod}\
W_\psi(\xi,\eta_\beta,0))
\end{array}$$
as required.

For (3), suppose $l=0$ and $\eta=\bar{0}$. We can write
$$uw=\!\!\!\!\!\!\!\!\sum_{|(\lambda,\mu,k)|\leq
N-2\varepsilon_2}\!\!\!\!\!\!\!\!x_{\lambda,\mu,k}f_{\lambda,\mu,k}(z)w
+\!\!\!\!\!\!\!\!\sum_{(\lambda,\bar{0},0)\in
\Lambda_{N-\varepsilon_2}}\!\!\!\!\!\!\!\!x_{1,\lambda}f_{\lambda,\bar{0},0}(z)w
+\!\!\!\!\!\!\sum_{\substack{(\lambda,\bar{0},0)\in
\Lambda_{N}\\
\lambda\prec\lambda'}}\!\!\!\!\!\!x_{1,\lambda}f_{\lambda,\bar{0},0}(z)w+x_{1,\lambda'}f_{\lambda',\bar{0},0}(z)w.$$
Let $\alpha=(0,\alpha_2)$ be the minimum element in $S_{\lambda'}$,
using the definitions of $\psi$ and
$W_\psi(\lambda_\alpha',\bar{0},0))$, we have
$$
\begin{array}{ll} &(d_2(\alpha+\varepsilon_2)-\psi(d_2(\alpha+\varepsilon_2)))uw\\
\vspace{2mm} \equiv & \sum_{\substack{(\lambda,\bar{0},0)\in
\Lambda_{N}\\
\lambda\prec\lambda'}}[d_2(\alpha+\varepsilon_2),x_{1,\lambda}f_{\lambda,\bar{0},0}(z)]w+[d_2(\alpha+\varepsilon_2),
x_{1,\lambda'}f_{\lambda',\bar{0},0}(z)]w\\
\vspace{2mm} \equiv &
[d_2(\alpha+\varepsilon_2),x_{1,\lambda'}]f_{\lambda',\bar{0},0}(z)w\\
\vspace{2mm} \equiv &
-\alpha_2\lambda'(\alpha)\psi(d_2(\varepsilon_2))
x_{1,\lambda_\alpha'}f_{\lambda',\bar{0},0}(z)w\\
\vspace{2mm} \equiv
&-\alpha_2\lambda'(\alpha)\psi(d_2(\varepsilon_2))
x_{\lambda_\alpha',\bar{0},0}f_{\lambda',\bar{0},0}(z)w
\quad(\textrm{mod}\ W_\psi(\lambda_\alpha',\bar{0},0))
\end{array}$$
as required.
\end{proof}

\vspace{3mm} {\bf{Proof of theorem 3.2}} \quad Suppose
$u=f(z)\in \mathbb{C}[z]$. For
$\alpha=(\alpha_1,\alpha_2)\in\mathbb{Z}^2_+$ and $i=1,2$, we have
 $[d_i(\alpha), z]=-\alpha_1 d_i(\alpha)$.
If $\alpha_1 >0$, by Proposition 2.1, we see that  $\psi(d_i(\alpha))=0$. Thus we have
$[d_i(\alpha),z]w=0$, that is $d_i(\alpha)zw=zd_i(\alpha)w$, which
follows that $d_i(\alpha)f(z)w=f(z)d_i(\alpha)w=\psi(d_i(\alpha))f(z)w$.
Therefore,
 $uw$ is a Whittaker vector of type $\psi$.

Suppose $w'=uw$ is a whittaker vector of  $W_{\psi}$. Using lemma
3.6, $w'$ is a whittaker vector of type $\psi$. Following from the
PBW-Theorem, we can write
$$uw=\sum_{(\lambda,\mu,k)\in \widetilde{\mathcal{P}}}
x_{\lambda,\mu,k}f_{\lambda,\mu,k}(z)w,\ \textrm{where }
f_{\lambda,\mu,k}(z)\in \mathbb{C}[z].$$ Let
$(\lambda',\mu',k')=\deg(uw)$. If $u\not\in \mathbb{C}[z]$, then
$(\lambda',\mu',k')\neq (\bar{0},\bar{0},0)$. Following from lemma
3.7--3.11, there is $D\in\mathcal{L}^+$ such that $(D-\psi(D))w'\neq
0$, which implies $w'$ is  not  a whittaker vector of  type $\psi$, a
contradition. Therefore, $u\in \mathbb{C}[z]$. This completes the proof of Theorem 3.2.

\noindent{\section{Simple  Whittaker modules of nonsingular type}}

Recall that a  module $V$ of a Lie algebra $\mathfrak{g}$ is said to be simple (or irreducible) if $V$ is nonzero and has
no nontrivial submodules. In this section, we study the submodule of
universal Whittaker module for $\mathcal{L}$ of nonsingular type. We
assume $\psi:\mathcal{L}^{+}\rightarrow\mathbb{C}$ is a nonsingular
homomorphism of Lie algebras,  $W_{\psi}$ is the universal Whittaker
module for $\mathcal{L}$ of type  $\psi$ with cyclic element $w$.

\begin{lemm}
Any nonzero submodule of    $W_{\psi}$ has a Whittaker vector.
\end{lemm}
\begin{proof} Suppose $M$ is a nonzero
submodule of $W_{\psi}$. Let $w'=uw$ be a nonzero element in $M$,
where $u\in U(\mathfrak{b}^{-})$. If $u\in \mathbb{C}[z]$, then by
theorem 3.2, $w'$ is a Whittaker vector as required. If $u\not\in
\mathbb{C}[z]$, then $\deg(w')\neq (\bar{0},\bar{0},0)$. Let
$(\lambda',\mu',k')=\deg(w')$. We can get a Whittaker vector in $M$
from $w'$ through the following steps.

\textbf{Step 1.}\quad If $k'>0$, then replace $w'$ by
$(d_2(2\varepsilon_2)-\psi(d_2(2\varepsilon_2)))^{k'}w'$. Following
from lemma 3.7, we have $w'\neq 0$ and $\deg(w')=(\lambda',\mu',0)$.
Then go to step 2.

\textbf{Step 2.}\quad If $\mu'\neq \bar{0}$ then $l(\mu')>0$. Let
$\alpha=(\alpha_1,\alpha_2)$   be the minimum element in $S_{\mu'}$.
If $\alpha_1>0$ then replace $w'$ by
$(d_1(\alpha+2\varepsilon_2)-\psi(d_1(\alpha+2\varepsilon_2)))w'$,
else then  replace $w'$ by
$(d_2(\alpha+2\varepsilon_2)-\psi(d_2(\alpha+2\varepsilon_2)))w'$.
Using lemma 3.8, we have $w'\neq 0$ and
$\deg(w')=(\lambda',\mu_\alpha',0)$. Note that
$l(\mu_\alpha')=l(\mu')-1$. If $\mu_\alpha'\neq \bar{0}$, repeat
what we have done as above, until we  get $\deg(w')=(\lambda',0,0)$.
Then go to step 3.

\textbf{Step 3.}\quad If $S_{\lambda'}^+\neq \emptyset$, then
replace $w'$ by
$(d_2(\alpha+2\varepsilon_2)-\psi(d_2(\alpha+2\varepsilon_2))) w'$,
where
  $\alpha$ is the minimum element in
  $S_{\lambda'}^+$. Using lemma 3.9, we have  $w'\neq 0$ and $\deg(w')=(\lambda_\alpha',\bar{0},0)$.
Note that  $\sum_{\beta\in
S_{\lambda_\alpha'}^+}\lambda'(\beta)=\sum_{\beta\in
S_{\lambda'}^+}\lambda'(\beta)$-1. If $S_{\lambda_\alpha'}^+\neq
\emptyset$, repeat what we have done as above, until we  get
$\deg(w')=(\xi,0,0)$ with $S_{\xi}^+=\emptyset$. Then go to step 4.

\textbf{Step 4.}\quad If $\xi\neq \bar{0}$, then there is a positive
integer $n$ such that $|\xi|=(0,n)$. Using lemma 3.10 and 3.11, we
can get $w''\in M$ such that $w''\neq 0$ and
$\deg(w'')=(\eta,\zeta,s)$, where $|(\eta,\zeta,s)|=(0,m)$ with
$0\leq m<n$. Replace $w'$ by $w''$. If
$(\eta,\zeta,s)=(\bar{0},\bar{0},0)$, then $w'$ is a Whittaker
vector in $M$ as required; else, then replace $(\lambda',\mu',k')$
by $(\eta,\zeta,s)$ and return to step 1.
\end{proof}

Let $I$ be an ideal of $\mathbb{C}[z]$, and $M(I)$ be a submodule of
$W_{\psi}$ which is generated by $Iw$. Since $\mathbb{C}[z]$ is a
principle ideal domain, there is  $g(z)\in I$ such that
$I=\mathbb{C}[z]g(z)$. Thus we have $M(I)=U(\mathfrak{b}^{-})g(z)w$.
In the following theorem, it is proved that any submodule of  $W_{\psi}$
is of this form.

\begin{theo} Let $M$ be a submodule of  $W_{\psi}$, then there exist  $g(z)\in \mathbb{C}[z]$
 such that  $M=U(\mathfrak{b}^{-})g(z)w$. Therefore, there is a bijection between the set of submodules of  $W_{\psi}$  and
the set of ideals of $\mathbb{C}[z]$. Moreover,  there is  a
bijection between the set of maximum submodules of $W_{\psi}$ and
the set of maximum ideals of $\mathbb{C}[z]$.
\end{theo}
\begin{proof}
Let $M(\neq\{0\})$ be a
submodule of $W_{\psi}$. Set  $\mathcal{I}_M=\{ f(z)\in
\mathbb{C}[z]\ | \ f(z)w\in M\}$.  Using lemma 4.1 and theorem 3.2,
we get $\mathcal{I}_M\neq\{0\}$. Clearly, $\mathcal{I}_M$ is an
ideal of $\mathbb{C}[z]$, thus $\mathcal{I}_M=\mathbb{C}[z]g(z)$ for
some $0\neq g(z)\in \mathbb{C}[z]$. We claim
$M=U(\mathfrak{b}^{-})g(z)w$. In fact, $M\supseteq
U(\mathfrak{b}^{-})g(z)w$ is obvious. Suppose $M\supsetneqq
U(\mathfrak{b}^{-})g(z)w$, then there is
 $0\neq v\in M$ such that $v\not\in U(\mathfrak{b}^{-})g(z)w$. 于是可设
$$v=\sum_{(\lambda,\mu,k)\in \widetilde{\mathcal{P}}}
x_{\lambda,\mu,k}f_{\lambda,\mu,k}(z)w,$$ where $
f_{\lambda,\mu,k}(z)\in \mathbb{C}[z]$  and
$\deg{f_{\lambda,\mu,k}(z)}<\deg{g(z)}$ if $f_{\lambda,\mu,k}(z)\neq
0$. Following from the proof of lemma 3.7--3.11, we can get $
h(z)w\in U(\mathfrak{b}^{-})v+U(\mathfrak{b}^{-})g(z)w \subseteq M$
such that $h(z)\neq 0$ and $\deg{h(z)}<\deg{g(z)}$. Thus $h(z)\in
\mathcal{I}_M$, a contradiction to
$\mathcal{I}_M=\mathbb{C}[z]g(z)$. This follows
$M=U(\mathfrak{b}^{-})g(z)w$.

Therefore, there is a bijection between  the set of submodules of
$W_{\psi}$ and the set of ideals of $\mathbb{C}[z]$. Moreover, there
is  a bijection between the set of maximum submodules of $W_{\psi}$
and the set of maximum ideals of $\mathbb{C}[z]$.
\end{proof}

\begin{coro} For  $a\in \mathbb{C}$,
   $L_{\psi,a}:=W_{\psi}/(U(\mathfrak{b}^{-})(z-a)w)$ is a simple Whittaker $\mathcal{L}$-module of type $\psi$,
  and any simple Whittaker $\mathcal{L}$-module of  type
$\psi$ is isomorphic to some $L_{\psi,a}$.
\end{coro}
\begin{proof}
We know that for  $a\in
\mathbb{C}$, $\mathbb{C}[z](z-a)$ is a maximum ideal of
$\mathbb{C}[z]$, and any maximum ideal of $\mathbb{C}[z]$ is of this
form. Using theorem 4.2, $U(\mathfrak{b}^{-})(z-a)w$ is a maximum
submodule of $W_{\psi}$, which implies
$L_{\psi,a}:=W_{\psi}/(U(\mathfrak{b}^{-})(z-a)w)$ is a simple
Whittaker $\mathcal{L}$-module of type $\psi$. Since any Whittaker
$\mathcal{L}$-module of type $\psi$ is isomorphic to a quotient of
 $W_{\psi}$, using theorem 4.2, it must be isomorphic to some $L_{\psi,a}$.
\end{proof}

\noindent {\bf Acknowledgments.} The authors would like to thank Prof. S. Tan  for  stimulating discussions and help in preparation of this paper.


\vskip 5mm\begin{thebibliography}{BS}

\bibitem[AP]{AP} D. Arnal, G. Pinczon,
On algebraically irreducible representations of the Lie algebra
$sl_2$.  J  Math Phys, 1974, 15: 350-359.

\bibitem[B]{B} R. Block, The irreducible representations of the Lie
 algebra $sl_2$ and of the Weyl algebra.
Adv. Math, 1981, 39: 69-110.



\bibitem[BO]{BO}  G. Benkart, M. Ondrus, Whittaker modules for
 generalized Weyl algebras. Represent Theory, 2009, 13: 141-164.



\bibitem[Ch]{Ch}  K. Christodoulopoulou, Whittaker modules for Heisenberg algebras
and imaginary Whittaker modules for affine Lie algebras.  J Algebra,
2008, 320: 2871-2890.

\bibitem[JM]{JM} C. Jiang, D. Meng, The automorphism group of the derivation algebra of the Virasoro-like algebra, Adv. Math. (China) 27
(2) (1998) 175-183.

\bibitem[LT1]{LT1} W. Lin, S. Tan, Representations of the Lie algebra of derivations for quantum torus. J Alg. 275 (2004) 250-274.

\bibitem[LT2]{LT2} W. Lin, S. Tan, Nonzero level Harish-Chandra modules over the Virasoro-like algebra, J. Pure Appl. Algebra 204 (2006)
90-105.

\bibitem[M]{M} O. Mathieu, Classification of Harish-Chandra modules over the Virasoro algebra, Invent. Math. 107 (1992) 225-234.

\bibitem[Ko]{Ko} B. Kostant, On Whittaker vectors and representation
theory. Invent Math, 1978, 48: 101-184.

\bibitem[KP]{KP} E. Kirkman, C. Procesi, L. Small, A q-analog for the Virasoro algebra, Comm. Algebra 22 (1994) 3755-3774.

\bibitem[OP]{OP} J. M. Osborn, D. S. Passman, Derivations of Skew Polynomial Rings. J. Algebra. 176 (1995) 417-448.

\bibitem[OW]{OW}  M. Ondrus , E. Wiesner,  Whittaker modules for the
Virasoro algebra. J. Algebra Appl, 2009, 8: 363-377.




\bibitem[WZ]{WZ} B. Wang, X. Zhu, Whittaker modules for a Lie algebra of Block type.
Front Math China, 2011, 6(4): 731-744.

\bibitem[WZh]{WZh} X. Wang, K. Zhao, Verma modules over the Virasoro-like algebra, J. Austral. Math. 80 (2006) 179-191.


\bibitem[ZTL]{ZTL}  X. Zhang, S. Tan, Lian H. Whittaker modules for the Schr\"{o}inger-Witt
algebra. J Math Phys, 2010, 51: 083524.

\bibitem[MZ1]{MZ1} V. Mazorchuk, K. Zhao, Classification of simple weight Virasoro modules
with a finite-dimensional weight space. J. Algebra, 307(2007):
209-214.

\bibitem[MZ2]{MZ2} V. Mazorchuk, K. Zhao, simple Virasoro modules which
are locally finite over a positive part, Selecta Math. (N.S.), 20  (2014),  no. 3, 839-854.

\bibitem[ZZ]{ZZ} H. Zhang and K. Zhao, Representations of the Virasoro-like Lie algebra and its
q-analog, Comm. Algebra  24(14)(1996), 4361-4372.

\end{thebibliography}
\end{document}